\documentclass[11pt]{amsart}
\usepackage{etex}
\usepackage{amssymb}
\usepackage{amscd}
\usepackage{amsthm}
\usepackage[backref=page]{hyperref}
\usepackage{extpfeil}
\usepackage{pictexwd,dcpic}
\usepackage[margin=1in]{geometry}
\usepackage{tikz}
\usetikzlibrary{matrix,arrows,decorations.pathmorphing}
\usepackage[all]{xy}

\usepackage{mathtools}

\numberwithin{equation}{section}
\theoremstyle{plain}
\newtheorem{thm}[equation]{Theorem}
\newtheorem{conj}[equation]{Conjecture}
\newtheorem{prop}[equation]{Proposition}

\theoremstyle{definition}
\newtheorem{defn}[equation]{Definition}

\theoremstyle{remark}
\newtheorem{rem}[equation]{Remark}
\renewcommand{\le}{\leqslant}
\renewcommand{\ge}{\geqslant}
\renewcommand{\setminus}{\smallsetminus}

\newcommand\blfootnote[1]{
  \begingroup
  \renewcommand\thefootnote{}\footnote{#1}
  \addtocounter{footnote}{-1}
  \endgroup}

\newcommand{\m}{\mathfrak{m}}

\newcommand{\Q}{\mathbb{Q}}

\newcommand{\C}{\mathbb{C}}

\newcommand \id {{\operatorname{id}}}
\newcommand\MF{{\operatorname{MF}}}
\newcommand\Hom{{\operatorname{Hom}}}

\newcommand\Spec{{\operatorname{Spec}}}

\newcommand\Perf{{\operatorname{Perf}}}

\newcommand{\fT}{\mathcal{T}}

\newcommand{\fC}{\mathcal{C}}

\newcommand{\fP}{\mathcal{P}}

\newcommand{\into}{\hookrightarrow}

\newcommand{\rKU}{\widetilde{KU}}

\newcommand{\ch}{\operatorname{ch}}

\newcommand{\coker}{\operatorname{coker}}

\newcommand{\Tor}{\operatorname{Tor}}

\newcommand{\Iso}{\operatorname{Iso}}

\newcommand{\codim}{\operatorname{codim}}
\newcommand{\Vect}{\operatorname{Vect}}
\newcommand{\Proj}{\operatorname{Proj}}

\def\on{\operatorname}
\def\onto{\twoheadrightarrow}

\begin{document}
\title{On a Conjecture of Dao-Kurano}
\author{Michael K. Brown}
\begin{abstract}
We prove a special case of a conjecture of Dao-Kurano concerning the vanishing of Hochster's theta pairing. The proof uses Adams operations on both topological $K$-theory and perfect complexes with support.
\end{abstract}
\maketitle

\tableofcontents

\blfootnote{\emph{Date:} \today}

\section{Introduction}
Let $A$ be a local hypersurface ring with maximal ideal $\eta$. Assume $A$ has an isolated singularity; that is, assume $A_\mathfrak{p}$ is a regular local ring for all $\mathfrak{p} \in \Spec(A) \setminus \{\eta\}$. If $M$ and $N$ are finitely generated $A$-modules, $l(\Tor_i^{A}(M,N)) < \infty$ for $i \gg 0$, where $l(-)$ denotes length as an $A$-module. Further, since minimal free resolutions of finitely generated $A$-modules are eventually 2-periodic, $\Tor_i^{A}(M,N) = \Tor_{i+2}^{A}(M,N)$ for $i \gg 0$. This motivates the following definition:
\begin{defn}
\label{theta}
Let $M$ and $N$ be finitely generated $A$-modules. The \emph{Hochster theta pairing} applied to $M$ and $N$ is given by 
$$\theta(M,N) = l(\Tor_{2i}^{A}(M,N)) - l(\Tor_{2i+1}^{A}(M,N)) \text{, } i \gg 0.$$
\end{defn}

The Hochster theta pairing was introduced by Hochster in \cite{hochster1981dimension}. Various conjectures concerning the vanishing of $\theta$ have received a good deal of attention lately: see, for instance, work of Buchweitz-van Straten (\cite{buchweitz2012index}), Dao (\cite{dao2013decent}), Moore-Piepmeyer-Spiroff-Walker (\cite{moore2011hochster}), and Polishchuk-Vaintrob (\cite{polishchuk2012chern}). For a more detailed history of the Hochster theta pairing, we refer the reader to Section 3 of Dao-Kurano's article \cite{DK12}.

Conjecture 3.1 of \cite{DK12} lists several open questions regarding the vanishing of $\theta$.  In particular, Dao-Kurano conjecture the following:

\begin{conj}[Dao-Kurano, \cite{DK12} Conjecture 3.1 (3)]
\label{conj}
Let $A$ be a local hypersurface of Krull dimension $d$ with isolated singularity, and let $M$ and $N$ be finitely generated $A$-modules. If $\dim(M) \le \frac{d}{2}$, $\theta(M, N) = 0$.
\end{conj}

Dao-Kurano themselves prove Conjecture~\ref{conj} in the following cases:
\begin{itemize}
\item $A = S_\eta$, where $S$ is a positively graded algebra over a field $k$ such that $\Proj(S)$ is smooth over $k$, and $\eta$ is the homogeneous maximal ideal of $S$
\item $A$ is excellent, $A$ contains a field of characteristic 0, and $d \le 6$
\end{itemize}

The goal of this paper is to prove Conjecture~\ref{conj} in the following additional special case:

\begin{thm}
\label{main}
Let $Q:= \C[x_1, \dots, x_n]$, let $\m$ denote the maximal ideal $(x_1, \dots, x_n) \subseteq Q$, let $f \in \m$, and let $R:=Q/(f)$. Assume the local hypersurface $R_\m$ has an isolated singularity. Let $M$ and $N$ be finitely generated $R_\m$-modules. If $\dim(M) \le \frac{n-1}{2}$, $\theta(M, N) = 0$.
\end{thm}

\begin{rem}
When $n$ is odd, Theorem~\ref{main} follows immediately from a theorem of Buchweitz-van Straten which implies that $\theta(M,N) = 0$ for \emph{all} finitely generated $R_\m$-modules $M,N$ (see the Main Theorem on page 245 of \cite{buchweitz2012index}).
\end{rem}

Our proof of the theorem uses Adams operations on both topological $K$-theory and perfect complexes of $Q$-modules with support in $V(f)$; the latter are introduced by Gillet-Soul\'e in \cite{GS87}. It is convenient here to use constructions of Adams operations involving cyclic actions on tensor powers of complexes of vector bundles and $Q$-modules, respectively. In topology, this approach is due to Atiyah in \cite{atiyah1966power}, and in \cite{brown2016cyclic}, the authors use Atiyah's ideas to carry out a similar construction in the setting of perfect complexes with support; see also Haution's Ph.D. thesis \cite{Haution}.

In Sections~\ref{top} and~\ref{complexes}, we discuss Adams operations on topological $K$-theory and perfect complexes with support, respectively. In Section~\ref{sectioncompatible}, we discuss a sense in which these operations are compatible; see Proposition~\ref{compatible} for the precise statement. In Section~\ref{proof}, we prove Theorem~\ref{main}.

{\bf  Acknowledgements.} I thank Claudia Miller, Peder Thompson, and Mark Walker for many conversations concerning Adams operations on perfect complexes with support. I also thank the Hausdorff Center for Mathematics, where most of this work was completed.

\section{Adams operations on topological $K$-theory}
\label{top}
We review some facts concerning Adams operations on topological $K$-theory. Everything in this section is classical and can be found in \cite{atiyah1966power}, Chapter 3 of \cite{atiyah1967k}, or \cite{segal1968equivariant}.

\subsection{Cyclic Adams operations}

Let $G$ be a topological group. If $Y$ is a compact Hausdorff $G$-space, let $\Vect_G(Y)$ denote the exact category of complex equivariant vector bundles on $Y$, and let $KU_G(Y)$ denote the Grothendieck group of $\Vect_G(Y)$. $KU_G(Y)$ is the \emph{complex equivariant $K$-theory group of $Y$}. When $G = \{*\}$, we write $KU(Y):= KU_G(Y)$.

Fix a connected compact Hausdorff space $X$ that is homotopy equivalent to a finite CW complex. The \emph{Adams operations} on $KU(X)$ are functions $\psi^k: KU(X) \to KU(X)$ for $k \ge 0$ with the following properties:
\begin{itemize}
\item The $\psi^k$ are additive
\item The $\psi^k$ are natural with respect to pullback along continuous maps
\item If $l$ is a line bundle on $X$, $\psi^k[l] = [l^{\otimes k}]$
\end{itemize}

These properties determine the operations $\psi^k$ uniquely, by the splitting principle. The eigenvalues of $\psi^k$ on $KU(X) \otimes \Q$ are all of the form $k^i$ for some $i \ge 0$; let $KU(X)^{(i)}$ denote the eigenspace corresponding to the eigenvalue $k^i$ in $KU(X) \otimes \Q$. The Chern character determines an isomorphism of graded $\Q$-vector spaces
$$\ch: \bigoplus_{i} KU(X)^{(i)} \xrightarrow{\cong} \bigoplus_i H^{2i}(X; \Q).$$

Let $\rKU(X)$ denote the \emph{reduced} topological $K$-theory of $X$, defined to be the cokernel of the map $KU(*) \to KU(X)$ induced by pullback along $X \to *$. Let $[n]$ denote the class in $KU(X)$ represented by the trivial bundle of rank $n$. A splitting of the short exact sequence
$$0 \to KU(*) \to KU(X) \xrightarrow{\pi} \rKU(X) \to 0$$
is given by the map $i: \rKU(X) \to KU(X)$ defined by $[V] \mapsto [V] - [\on{rank}(V)]$. We define operations
on $\rKU(X)$ by $\pi \circ \psi^k \circ i$, and we denote them also by $\psi^k$.

Let $p$ be a prime. We recall a construction of $\psi^p: KU(X) \to KU(X)$ due to Atiyah in \cite{atiyah1966power}. Given a vector bundle $V$ over $X$, the $p$-th tensor power $V^{\otimes p}$ of $V$ may be equipped with a canonical action of the symmetric group on $p$ letters, and hence determines a class in $KU_{\Sigma_p}(X)$ (where $X$ is equipped with the trivial action of $\Sigma_p$). By Proposition 2.2 of \cite{atiyah1966power}, there exists a map $KU(X) \to KU_{\Sigma_p}(X)$
sending a class $[V]$ represented by a bundle $V$ to $[V^{\otimes p}]$. Let $C_p$ denote the cyclic subgroup of $\Sigma_p$ generated by $\sigma_p:=(1 2 \dots p)$; the inclusion $C_p \into \Sigma_p$ determines a map $KU_{\Sigma_p}(X) \to KU_{C_p}(X)$. Let $t^p: KU(X) \to KU_{C_p}(X)$ denote the composition of these two maps.

Since $C_p$ acts trivially on $X$, one has a canonical isomorphism
$$KU_{C_p}(X) \xrightarrow{\cong} KU(X) \otimes R(C_p),$$
where $R(C_p)$ denotes the representation ring of $C_p$ (\cite{segal1968equivariant} Proposition 2.2). On classes of the form $[V]$, where $V$ is a vector bundle, the isomorphism is given by
$$[V] \mapsto \sum_{j=0}^{p-1} \Hom_{C_p}(M_j, V) \otimes W_j,$$
where $W_j$ is the irreducible representation of $C_p$ corresponding to the character $\sigma_p \mapsto e^{2 \pi i j / p}$, and $M_j$ is the $C_p$-bundle $W_j \times X$. By line (2.7) of \cite{atiyah1966power}, if $V$ is a vector bundle on $X$, 
$$\psi^p([V]) = [\Hom_{C_p}(M_0, t^p[V])] - [\Hom_{C_p}(M_1, t^p[V])].$$

\subsection{Variant for relative $K$-theory}
Let $X$ be a connected compact Hausdorff space, and let $Y$ be a closed subspace of $X$ such that $(X,Y)$ is homotopy equivalent to a finite CW pair. Let $G$ be a topological group, suppose $X$ is equipped with an action of $G$, and suppose $Y$ is a $G$-subspace of $X$. Let $\fC_G(X,Y)$ denote the exact category of bounded complexes of complex equivariant vector bundles on $X$ whose restrictions to $Y$ are exact. 

\begin{defn}
Objects $C_0, C_1$ of $\fC_G(X,Y)$ are said to be \emph{homotopic} if there exists an object $C$ of $\fC_G(X \times [0, 1], Y \times [0,1])$ such that the restriction of $C$ to $X \times \{i\}$ is isomorphic to $C_i$ for $i = 0,1$. \end{defn}

Set $\Iso(\fC(X,Y))$ to be the monoid of isomorphism classes in $\fC_G(X,Y)$ with operation $\oplus$. If $[C_0], [C_1] \in \Iso(\fC(X,Y))$, we say $[C_0] \sim [C_1]$ if and only if there exist exact complexes $E_0, E_1 \in \fC_G(X,Y)$ such that $C_0 \oplus E_0$ is homotopic to $C_1 \oplus E_1$. Define $L_G(X,Y)$ to be the monoid $\Iso(\fC(X,Y)) / \sim$.

Let $KU_G(X,Y)$ denote the relative equivariant $K$-theory group of the pair $(X,Y)$ (\cite{segal1968equivariant} Definition 2.8).

\begin{thm}[\cite{segal1968equivariant} Section 3]
There exists a natural isomorphism 
$$\chi: L_G(X,Y)  \xrightarrow{\cong} KU_G(X,Y).$$ 
In particular, $L_G(X,Y)$ is a group.
\end{thm}

When $Y = \emptyset$, $\chi$ is given by the Euler characteristic. Note that $L_G(X,Y)$ has a product operation $\otimes$ induced by tensor product of complexes, giving it the structure of a non-unital ring, and $\chi(C_0 \otimes C_1) = \chi(C_0) \otimes \chi(C_1)$, where the product on the right-hand side is induced by tensor product of $G$-bundles.

If $G = \{*\}$, we write $L(X,Y):= L_G(X,Y)$ and $ KU(X,Y):=KU_G(X,Y)$. By definition, $KU(X,Y) = \rKU(X/Y)$. Fix a prime $p$. We now wish to provide an alternative construction, via $\chi$, of the operations $\psi^p$ on $KU(X,Y)$ defined above.

By Section 3 of \cite{atiyah1966power}, the $p$-th tensor power of complexes induces a homomorphism 
$$L(X,Y) \to L_{\Sigma_p}(X,Y),$$ 
where, in the target, $X$ is equipped with the trivial action of $\Sigma_p$. Restricting along the inclusion $C_p \into \Sigma_p$, we obtain a map
$$t^p: L(X,Y) \to L_{C_p}(X,Y).$$

Given complexes $C, C'$ of $C_p$-bundles on $X$, let $\Hom_{C_p}(C, C')$ denote the morphism complex in the category of complexes of $C_p$-bundles on $X$. Let $M_0, \dots, M_{p-1}$ be as defined above, considered as complexes of $C_p$-bundles concentrated in degree 0. For each $j$, $\Hom_{C_p}(M_j, -)$ yields an exact functor $\fC_{C_p}(X,Y) \to \fC(X,Y)$, and it preserves homotopy; thus, $\Hom_{C_p}(M_j, -)$ induces a map $L_{C_p}(X,Y) \to L(X,Y)$. Define 
$$\psi^p: L(X,Y) \to L(X,Y)$$
to be given by 
$$\psi^p([C]) = [\Hom_{C_p}(M_0, t^p[C])] - [\Hom_{C_p}(M_1, t^p[C])],$$
Since $\chi$ is multiplicative, it is not hard to check that one has a commutative square
\begin{equation}
\label{compat}
\xymatrix{ L(X,Y) \ar[r]^-{\chi} \ar[d]^-{\psi^p} & KU(X,Y) \ar[d]^-{\psi^p}  \\
L(X,Y) \ar[r]^-{\chi} & KU(X,Y)
\\}
\end{equation}

\section{Adams operations on perfect complexes with support}
\label{complexes}

Let $Q$ be a commutative Noetherian $\C$-algebra, let $Z \subseteq \Spec(Q)$ be a closed subset, and let $G$ be a finite group. Let $\fP^Z(Q; G)$ denote the category of bounded complexes of finitely generated projective $Q$-modules with support in $Z$ and equipped with
a left $G$-action (with $G$ acting via chain maps). Let $K^Z_0(Q; G)$ denote the Grothendieck group of $\fP^Z(Q; G)$, defined to be the
group generated by isomorphism classes of objects modulo the relations
$$[X] = [X'] + [X''] $$
if there exists an (equivariant) short exact sequence $0 \to X' \to X \to X''$, and
$$[X] = [Y]$$
if there exists an (equivariant) quasi-isomorphism joining $X$ and $Y$; the group operation is given by direct sum. When $G = \{*\}$, we write $K_0^Z(Q) := K^Z_0(Q; G)$.

In this section, we recall a construction of Adams operations on $K_0^Z(Q)$ involving cyclic actions on tensor powers of complexes, following Section 3 of \cite{brown2016cyclic}; see also Haution's Ph.D. thesis \cite{Haution} for a similar discussion. The construction is inspired by the cyclic Adams operations on topological $K$-theory, due to Atiyah, which we discuss in the previous section.

Adams operations on $K_0^Z(Q)$ were introduced by Gillet-Soul\'e in \cite{GS87} for the purpose of proving Serre's Vanishing Conjecture. By Corollary 6.14 of \cite{brown2016cyclic}, the operations we discuss here agree, in our setting, with those constructed by Gillet-Soul\'e.

Let $p$ be a prime. If $X$ is an object of $\fP^Z(Q)$, the $p$-th tensor power $X^{\otimes p}$ has a canonical signed left action of $\Sigma_p$. By Theorem 2.2 of \cite{brown2016cyclic}, there exists a map $K_0^Z(Q) \to K_0^Z(Q; \Sigma_p)$ 
that sends a class $[X]$ represented by an object $X$ in $\fP^Z(Q)$ to $[X^{\otimes p}]$. Restricting along the inclusion $C_p \into \Sigma_p$ yields a map
$$t^p : K_0^Z(Q) \to K_0^Z(Q; C_p).$$

For $0 \le j \le p-1$, let $Q_j$ denote the projective $Q[C_p]$-module $Q$ with action $\sigma_p q = e^{2\pi ij/p} q$ (recall that $\sigma_p = (12 \dots p)$). Define the \emph{$p$-th Adams operation} $\psi^p : K_0^Z(Q) \to K_0^Z(Q)$ by 
$$Y \mapsto [\Hom_{Q[C_p]}(Q_0, t^p(Y))] - [\Hom_{Q[C_p]}(Q_1, t^p(Y))].$$
By Theorem 3.7 of \cite{brown2016cyclic}, $\psi^p$ is a group endomorphism. We refer the reader to \cite{brown2016cyclic} for a thorough discussion of the properties of the Adams operations $\psi^p$. One such property we will need later on is: 

\begin{thm}[\cite{brown2016cyclic} Corollary 3.12] 
\label{eigen}
If $p$ is prime, $Q$ is regular of Krull dimension $d$, and $Z$ has codimension $c$ in $\Spec(Q)$, there exists a direct sum decomposition
$$K_0^Z(Q) \otimes \Q = \bigoplus_{i = c}^d K_0^Z(Q)^{(i)}$$
where $K_0^Z(Q)^{(i)}$ is the eigenspace of $\psi^p$ in $K_0^Z(Q) \otimes \Q$ corresponding to the eigenvalue $p^i$. Moreover, if $M$ is a finitely generated $Q$-module supported on $Z$, and $X$ is a finite projective resolution of $M$, one has 
$$[X] \in \bigoplus_{i = \codim_Q M}^d K_0^Z(Q)^{(i)}.$$
\end{thm}

\begin{rem}
The idea of the proof of this theorem is essentially due to Gillet-Soul\'e in \cite{GS87}. They show that the above theorem holds for any family of operations on $K$-theory with supports satisfying conditions A1) through A4) in Section 4.11 of \cite{GS87}. Thus, the authors of \cite{brown2016cyclic} need only show that the operations $\psi^k$ defined above satisfy these conditions, and they prove this in Theorem 3.7 of \cite{brown2016cyclic} (note that Theorem 3.7 of \cite{brown2016cyclic} is proven only in the setting of affine schemes, but this is enough to conclude the above theorem; see Remark 3.8 of \cite{brown2016cyclic}).
\end{rem}

\section{Compatibility of Adams operations}
\label{sectioncompatible}
Let $Q:=\C[x_1, \dots, x_n]$, let $\m:=(x_1, \dots, x_n)$, and let $f \in \m \setminus \{0\}$. The goal of this section is to exhibit a precise sense in which the Adams operations on $K_0^{V(f)}(Q)$ are compatible with Adams operations on topological $K$-theory.

\subsection{The Milnor fibration}
\label{milnor}

By well-known theorems of L\^e and Milnor (\cite{le1976some}, \cite{milnor1968singular}), there exist $\epsilon > 0$ and $0 < \delta \ll \epsilon$ such that, if
\begin{itemize} 
\item $B \subseteq \C^n$ denotes the closed ball centered at the origin of radius $\epsilon'$, where $0 < \epsilon' < \epsilon$, and
\item $D^* \subseteq \C$ denotes the open disk of radius $\delta'$ centered at the origin with the origin removed, where $0 < \delta' < \delta$,
\end{itemize}
the map $\psi: B \cap f^{-1}(D^*) \to D^*$ given by $x \mapsto f(x)$ is a fibration, called the \emph{Milnor fibration} of $f$. Choose such $\epsilon'$ and $\delta'$, choose $t \in D^*$, and set $F:=\psi^{-1}(t)$. $F$ is called the \emph{Milnor fiber} of $f$. $F$ is independent of the choices of $\epsilon', \delta',$ and $t$ up to homotopy equivalence.

We refer the reader to Chapter 3 of Dimca's text \cite{dimca2012singularities} for a detailed discussion of the Milnor fibration. We point out a key property of the Milnor fiber which we will use later on. Set
$$\mu:= \dim_\C \frac{\C[x_1, \dots, x_n]_\m}{(\frac{\partial f}{\partial x_1}, \dots, \frac{\partial f}{\partial x_n})},$$
the \emph{Milnor number} of $f$ at the origin. The following is a famous theorem of Milnor:

\begin{thm}[\cite{milnor1968singular} Theorem 6.5]
\label{bouquet}
If $\mu < \infty$, $F$ is homotopy equivalent to a wedge sum of $\mu$ copies of $S^{n-1}$.
\end{thm}

\subsection{Matrix factorizations}
\label{mf}
We recall some background on matrix factorizations in commutative algebra. Let $S$ be a commutative ring, and let $w \in S$.

\begin{defn} A \emph{matrix factorization of $w$ over $S$} is a pair of finitely generated projective $S$-modules $F_0, F_1$ equipped with maps
$$d_0: F_0 \to F_1 \text{, } d_1: F_1 \to F_0$$
such that $d_1 d_0 = w \cdot \id_{F_0}$ and $d_0 d_1 = w \cdot \id_{F_1}$. We denote matrix factorizations by $(F_0, F_1, d_0, d_1)$. 
\end{defn}

One may form the \emph{homotopy category of matrix factorizations} $[\MF(S, w)]$ with objects given by matrix factorizations of $w$ over $S$; see Definition 2.1 of \cite{dyckerhoff2011compact} for the definition of $[\MF(S, w)]$. 

Assume $S$ is regular of finite Krull dimension and $w$ is a non-unit, non-zero-divisor of $S$. In this case, $[\MF(S, w)]$ may be equipped with a canonical triangulated structure; see Section 3.1 of \cite{orlov2003triangulated} for details. In fact, setting $T:=S/(w)$, there exists an equivalence of triangulated categories
$$[\MF(S,w)] \xrightarrow{\cong} \on{D}^b(T) / \Perf(T),$$
where the right-hand side is the Verdier quotient of the bounded derived category of $T$ by the triangulated subcategory consisting of perfect complexes (\cite{orlov2003triangulated} Theorem 3.9). The equivalence sends a matrix factorization $(F_0, F_1, d_0, d_1)$ to the complex with $\coker(d_1)$ concentrated in degree 0.

\subsection{Compatibility}
\label{gamma}
 
Set $R:=Q/(f)$, where $Q$, $f$ are as in the beginning of this section. Our next goal is to define a group homomorphism 
$$\gamma: K_0^{V(f)}(Q) \to L(B,F),$$
that is compatible with Adams operations, where $B, F$ are as defined in Section~\ref{milnor}. We proceed as follows:
\begin{itemize}
\item By Lemma 1.9 of \cite{GS87}, there exists an isomorphism $r: G_0(R) \xrightarrow{\cong} K_0^{V(f)}(Q)$ that sends a class represented by a module $M$ to the class represented by a finite $Q$-free resolution of $M$. 

\item We recall that the Grothendieck group of a triangulated category $\fT$ is the free abelian group on isomorphism classes of objects in $\fT$ modulo relations given by exact triangles. Let $K_0[\MF(Q,f)]$ denote the Grothendieck group of the triangulated category $[\MF(Q,f)]$. The equivalence $[\MF(Q,f)] \xrightarrow{\cong} \on{D}^b(R) / \Perf(R)$ discussed in Section~\ref{mf} yields an isomorphism $K_0[\MF(Q,f)] \xrightarrow{\cong} G_0(R) / \on{im}(K_0(R) \to G_0(R))$. In particular, we have a surjection 
$$s: G_0(R) \onto K_0[\MF(Q,f)].$$

\item Let $E = (F_0, F_1, d_0, d_1)$ be a matrix factorization of $f$ over $Q$. The following construction, introduced by Buchweitz-van Straten in \cite{buchweitz2012index}, associates a class in $L(B, F)$ to the matrix factorization $E$.

Denote by $C(B)$ the ring of
$\C$-valued continuous functions on $B$. Applying extension of scalars along the inclusion
$$Q \into C(B),$$
we obtain a map
$$F_1 \otimes_Q C(B) \xrightarrow{d_1 \otimes \id} F_0
\otimes_Q C(B)$$
of finitely generated free $C(B)$-modules. The category of complex vector bundles over $B$ is
equivalent to the category of finitely generated free
$C(B)$-modules; on objects, the equivalence sends a bundle to
its space of sections.
Let 
$$V_1 \xrightarrow{d_1} V_0$$
be a map of vector bundles over $B$ corresponding to the above map $d_1 \otimes \id$ under this
equivalence. Recall that $t \in \C$ is the value over which we defined the Milnor fiber $F$. Since $d_1 \circ d_0 = f \cdot \id_{F_0}$ and $d_0 \circ d_1 = f \cdot \id_{F_1}$, and since the restriction of the polynomial $f$, thought of as a map $\C^n \to \C$, to $F = B \cap f^{-1}(t)$ is constant with value $t \ne 0$, $d_1|_{F}$ is an isomorphism of vector bundles on $F$. Its inverse is the
restriction to $F$ of the map $V_0 \to V_1$ determined by 
$$F_0
\otimes_Q C(B) \xrightarrow{\frac{1}{t} (d_0 \otimes \id)}
F_1 \otimes_Q C(B).$$
Define $\Phi(E)$ to be the class in $L(B,F)$ represented by the complex $0 \to V_1 \xrightarrow{d_1} V_0 \to 0$. By (the complex version of) Proposition 3.19 of \cite{brown2015kn}, Buchweitz-van Straten's construction $E \mapsto \Phi(E)$ induces a group homomorphism
$$\phi: K_0[\MF(Q,f)] \to L(B,F).$$
\end{itemize}

Finally, we define $\gamma: K_0^{V(f)}(Q) \to L(B,F)$ to be the composition
$$K_0^{V(f)}(Q) \xrightarrow{r^{-1}} G_0(R) \xrightarrow{s} K_0[\MF(Q, f)] \xrightarrow{\phi} L(B,F).$$

Now, suppose $0 \to F_1 \xrightarrow{d_1} F_0 \to 0$ is a $Q$-projective resolution of an $R$-module $M$. Then $r^{-1}([0 \to F_1 \xrightarrow{d_1} F_0 \to 0]) = [M]$, and $s([M])$ is of the form $[F_1, F_0, d_1, d_0]$ for some map $d_0: F_0 \to F_1$. Let $V_1 \xrightarrow{d_1} V_0$ be a map of vector bundles over $B$ corresponding to the map 
$$F_1 \otimes_Q C(B) \xrightarrow{d_1 \otimes \id} F_0
\otimes_Q C(B)$$ 
of free $C(B)$-modules, as in the third bullet above. Then one has 
$$\gamma([0 \to F_1 \xrightarrow{d_1} F_0 \to 0]) = [0 \to V_1 \xrightarrow{d_1} V_0 \to 0].$$

Using the isomorphisms 
$$r: G_0(R) \xrightarrow{\cong} K_0^{V(f)}(Q) \text{, } K_0[\MF(Q,f)] \xrightarrow{\cong} G_0(R)/\on{im}(K_0(R) \to G_0(R)),$$ it is easy to see that classes of the form $[P]$, where $P$ is a two-term $Q$-free resolution of an $R$-module $M$, generate $K_0^{V(f)}(Q)$ as an abelian group. Thus, the following is immediate from the constructions of the Adams operations on $K_0^{V(f)}(Q)$ and $L(B, F)$ discussed above:

\begin{prop}
\label{compatible}
If $p$ is prime and $X \in K_0^{V(f)}(Q)$, $\gamma(\psi^p(X)) = \psi^p(\gamma(X))$.
\end{prop}

\begin{rem}
\label{g}
Let $g$ be an element of $Q$ such that $g \notin \m$. Suppose that, in the construction of the Milnor fiber, $\epsilon'$ is chosen to be so small that $B \cap g^{-1}(0) = \emptyset$. Then one may define maps
$$r_g: G_0(R_g) \xrightarrow{\cong} K_0^{V(f)}(Q_g) \text{, } s_g: G_0(R_g) \onto K_0[\MF(Q_g, f)], \phi_g: K_0[\MF(Q_g, f)] \to L_1(B,F)$$
in exactly the same way as in the three bullets above. Set $\gamma_g:= \phi_g s_g  r^{-1}_g$. If $p$ is prime, one has $\gamma_g  \psi^p  = \psi^p \gamma_g$, by the same reasoning as above.
\end{rem}

\section{Proof of Theorem~\ref{main}}
\label{proof}
\begin{proof}[Proof of Theorem~\ref{main}]
Choose $g \notin \m$ such that $R_g$ has an isolated singularity only at $\m$; that is, such that $(R_g)_\mathfrak{p}$ is regular for all $\mathfrak{p} \in \Spec(Q_g) \setminus \m$. Without loss of generality, assume that, in the construction of the Milnor fiber $F$ in Section~\ref{milnor}, $\epsilon'$ is chosen to be so small that $B \cap g^{-1}(0) = \emptyset$. Let $r_g, s_g, \phi_g$, and $\gamma_g$ be defined as in Remark~\ref{g}. Also, let $s_\m: G_0(R_\m) \onto K_0[\MF(Q_\m, f)]$ denote the surjection defined in the same way as the map $s$ in the second bullet of Section~\ref{gamma}. 

By Theorem 4.11 of \cite{dyckerhoff2011compact}, the functor $[\MF(Q_g, f)] \to [\MF(Q_\m, f)]$
induced by extension of scalars along the localization map $Q_g \to Q_\m$ is an equivalence; let $l: K_0[\MF(Q_g, f)] \xrightarrow{\cong} K_0[\MF(Q_\m, f)]$ denote the induced isomorphism on Grothendieck groups. By Propositions 4.1 and 4.2 of \cite{buchweitz2012index}, if $(\phi_g \circ l^{-1} \circ s_\m)([M]) = 0$, then $\theta(M,N) =0.$ Let $M'$ be an $R_g$-module such that $s_g([M']) = (l^{-1} \circ s_\m)([M])$, and let $P$ be a finite $Q_g$-free resolution of $M'$. It suffices to show $\gamma_g([P]) = 0$.

Let $p$ be a prime. Define
\begin{displaymath}
   m=\left\{
     \begin{array}{ll}
        \frac{n}{2} +1 &  \text{if } n \text{ is even}\\
       \frac{n+1}{2}  &  \text{if } n \text{ is odd} \\
           \end{array}
   \right.
\end{displaymath} 
By Theorem~\ref{eigen}, $[P]  \in \bigoplus_{i = m}^{n} K_0^{V(f)}(Q_g)^{(i)}.$ Thus, Remark~\ref{g} and the commutativity of diagram (\ref{compat}) imply
$$(\ch \circ \chi \circ \gamma_g)([P]) \in \bigoplus_{i = m}^{n} H^{2i}(B/F ; \Q) \cong  \bigoplus_{i = m}^{n} H^{2i}(\Sigma F ; \Q),$$
where $\Sigma F$ denotes the suspension of $F$. By Theorem~\ref{bouquet}, $H^i(\Sigma F; \Q) = 0$ when $i > n$; thus, $(\ch \circ \chi \circ \gamma_g)([P]) = 0$, and so $\gamma_g([P]) = 0$.

\end{proof}
\bibliographystyle{amsalpha}
\bibliography{Bibliography}

\providecommand{\bysame}{\leavevmode\hbox to3em{\hrulefill}\thinspace}
\providecommand{\MR}{\relax\ifhmode\unskip\space\fi MR }
\providecommand{\MRhref}[2]{%
  \href{http://www.ams.org/mathscinet-getitem?mr=#1}{#2}
}
\providecommand{\href}[2]{#2}
\begin{thebibliography}{BMTW16}

\bibitem[Ati66]{atiyah1966power}
Michael~Francis Atiyah, \emph{Power operations in {K}-theory}, The Quarterly
  Journal of Mathematics \textbf{17} (1966), no.~1, 165--193.

\bibitem[Ati67]{atiyah1967k}
MF~Atiyah, \emph{K-theory}, vol.~2, WA Benjamin New York, 1967.

\bibitem[BMTW16]{brown2016cyclic}
Michael~K Brown, Claudia Miller, Peder Thompson, and Mark~E Walker,
  \emph{Cyclic {A}dams operations}, arXiv preprint arXiv:1601.05072 (2016).

\bibitem[Bro15]{brown2015kn}
Michael~K Brown, \emph{Kn\" orrer periodicity and {B}ott periodicity}, arXiv
  preprint arXiv:1507.03329 (2015).

\bibitem[BvS12]{buchweitz2012index}
Ragnar-Olaf Buchweitz and Duco van Straten, \emph{An index theorem for modules
  on a hypersurface singularity}, Mosc. Math. J \textbf{12} (2012), no.~2,
  237--259.

\bibitem[Dao13]{dao2013decent}
Hailong Dao, \emph{Decent intersection and {T}or-rigidity for modules over
  local hypersurfaces}, Transactions of the American Mathematical Society
  \textbf{365} (2013), no.~6, 2803--2821.

\bibitem[Dim12]{dimca2012singularities}
Alexandru Dimca, \emph{Singularities and topology of hypersurfaces}, Springer
  Science \& Business Media, 2012.

\bibitem[DK14]{DK12}
Hailong Dao and Kazuhiko Kurano, \emph{Hochster's theta pairing and numerical
  equivalence}, Journal of K-theory: K-theory and its Applications to Algebra,
  Geometry, and Topology \textbf{14} (2014), no.~03, 495--525.

\bibitem[Dyc11]{dyckerhoff2011compact}
Tobias Dyckerhoff, \emph{Compact generators in categories of matrix
  factorizations}, Duke Mathematical Journal \textbf{159} (2011), no.~2,
  223--274.

\bibitem[GS87]{GS87}
H.~Gillet and C.~Soul\'e, \emph{Intersection theory using {A}dams operations},
  Inventiones Mathematicae \textbf{90} (1987), 243--277.

\bibitem[Hau09]{Haution}
Olivier Haution, \emph{Steenrod operations and quadratic forms}, Ph.D. thesis,
  2009.

\bibitem[Hoc81]{hochster1981dimension}
Melvin Hochster, \emph{The dimension of an intersection in an ambient
  hypersurface}, Algebraic geometry, Springer, 1981, pp.~93--106.

\bibitem[L{\^e}76]{le1976some}
D{\~u}ng~Tr{\'a}ng L{\^e}, \emph{Some remarks on relative monodromy}, Centre de
  math{\'e}matiques de l'{\'E}cole polytechnique, 1976.

\bibitem[Mil68]{milnor1968singular}
John~Willard Milnor, \emph{Singular points of complex hypersurfaces}, no.~61,
  Princeton University Press, 1968.

\bibitem[MPSW11]{moore2011hochster}
W~Frank Moore, Greg Piepmeyer, Sandra Spiroff, and Mark~E Walker,
  \emph{{H}ochster's theta invariant and the {H}odge--{R}iemann bilinear
  relations}, Advances in Mathematics \textbf{226} (2011), no.~2, 1692--1714.

\bibitem[Orl03]{orlov2003triangulated}
Dmitri Orlov, \emph{Triangulated categories of singularities and {D}-branes in
  {L}andau-{G}inzburg models}, arXiv preprint math/0302304 (2003).

\bibitem[PV12]{polishchuk2012chern}
Alexander Polishchuk and Arkady Vaintrob, \emph{{C}hern characters and
  {H}irzebruch--{R}iemann--{R}och formula for matrix factorizations}, Duke
  Mathematical Journal \textbf{161} (2012), no.~10, 1863--1926.

\bibitem[Seg68]{segal1968equivariant}
Graeme Segal, \emph{Equivariant {$ K $}-theory}, Publications Math{\'e}matiques
  de l'IH{\'E}S \textbf{34} (1968), 129--151.

\end{thebibliography}

\end{document}